\documentclass[12pt]{amsart}

\usepackage{stackrel}
\DeclareMathOperator{\Li}{Li}

\newtheorem{theorem}{Theorem}[section]
\newtheorem{lemma}{Lemma}[section]
\newtheorem{prop}[theorem]{Proposition}

\newtheorem{cor}[theorem]{Corollary}

\newtheorem{remark}{Remark}[section]

\numberwithin{equation}{section}

\newcommand{\lt}{\left}
\newcommand{\rt}{\right}
\newcommand{\bpm}{\begin{pmatrix}}
\newcommand{\epm}{\end{pmatrix}}
\newcommand{\bsm}{\lt(\begin{smallmatrix}}
\newcommand{\esm}{\end{smallmatrix}\rt)}
\newcommand{\beq}{\begin{equation}}
\newcommand{\eeq}{\end{equation}}
\newcommand{\bal}{\begin{align}}

\newcommand{\vep}{\varepsilon}
\newcommand{\ep}{\epsilon}
\newcommand{\mf}{\mathfrak}

\newcommand{\hf}{\frac{1}{2}}

\newcommand{\mc}{\mathcal}

\usepackage[usenames,dvipsnames]{xcolor}

\title{Bertrand's postulate for Number Fields}
\author{Thomas A. Hulse}
\address{ Thomas A. Hulse \newline
\indent Department of Mathematics and Statistics \newline
\indent Colby College \newline
\indent Waterville, ME, USA \newline
\indent 04901}
       \email{tahulse@colby.edu}
\thanks{Research of the first author was partially supported by a Coleman Postdoctoral Fellowship at Queen's University.}
\author{M. Ram Murty}
\address{M. Ram Murty \newline
\indent Department of Mathematics and Statistics \newline
\indent Queen's University \newline
\indent Kingston, ON, Canada \newline
\indent K7L 3N6
}
       \email{murty@mast.queensu.ca}
\thanks{Research of the second author was partially supported by an NSERC Discovery Grant.}
\date{\today}
\begin{document}
\maketitle
\begin{abstract}
Consider an algebraic  number field, $K$, and its ring of integers, $\mc{O}_K$. There exists a smallest $B_K>1$ such that for any $x>1$ we can find a prime ideal, $\mf{p}$, in $\mc{O}_K$ with norm $N(\mf{p})$ in the interval $[x,B_Kx]$. This is a generalization of Bertrand's postulate to number fields, and in this paper we produce bounds on $B_K$ in terms of the invariants of $K$ from an effective prime ideal theorem due to Lagarias and Odlyzko \cite{LO}. We also show that a bound on $B_K$ can be obtained from an asymptotic estimate for the number of ideals in $\mc{O}_K$ less than $x$. 
\end{abstract}

\section{Introduction}


Predating the prime number theorem, Bertrand's postulate was first put forward by Joseph Bertrand in 1845 and proved
by Chebyshev in 1850. It states that, for any $x>1$, a prime number can be found in the interval $[x,2x]$. This is generally considered to be a much weaker result than the prime number theorem, as one can use the asymptotic behavior of the prime counting function, $\pi(x)$, to show that for any $A>1$ there exists $x_A>1$ such that for any $x> x_A$ we have $\pi(Ax)-\pi(x)>0$, and so there is a prime number in the interval $[x,Ax]$. Thus, in principle, one can use the prime number theorem to bound $x_A$ from above and find a lowest possible $x_A$ by employing a finite search. Indeed, Betrand's postulate itself can be recovered using more precise upper and lower bounds for $\pi(x)$, like those due to Dusart \cite{D} which arise from numerical verification of the Riemann hypothesis for the first $1.5 \cdot 10^9$ zeros.


Outside of historical interest, however, one of the main benefits of Bertrand's postulate is that it gives information about the distribution of primes when $x$ is small. Furthermore, it accomplishes this without requiring information about the zeros of the Riemann zeta function. Indeed, Bertrand's postulate benefits from having many short and often elegant, elementary proofs  \cite{E,Rj,MM}.

One may similarly investigate a variant of Bertrand's postulate for the distribution of prime ideals in the ring of integers, $\mc{O}_K$, of an algebraic number field, $K$. That is, for general $K$, we ask if we can find $B$ such that for all $x>1$ there exists a prime ideal $\mathfrak{p}$ in $\mc{O}_K$ with norm $N(\mathfrak{p}) \in [x,Bx]$. Indeed, if such a $B$ exists for any given number field we can define the \emph{Bertrand constant}, $B_K$, to be the best such $B$,
\beq
B_K := \min \left\{B >1 \ | \ \forall x>1, \  \exists \  \mf{p} \subseteq \mc{O}_K, \  N(\mathfrak{p}) \in [x,Bx] \right\}.
\eeq
We can define $B_K$ to be a minimum instead of an infimum, for if $B_K$ is the infimum of the non-empty set then for any $x>1$ we can find a prime ideal with norm in $[x,(B_K+\vep)x]$ for any $\vep>0$. Since the norms of ideals are rational integers, we can take $\vep>0$ small enough so that there must be a prime ideal with norm in $[x,B_K x]$, and thus $B_K$ is an element of the above set.

We know that $B_K$ must exist due to the prime ideal theorem, first proven by Landau in 1903, which states that for $x>1$,
\beq \label{pit}
\pi_K(x) \sim \frac{x}{\log x},
\eeq
where $\pi_K(x)$ counts the number of prime ideals in $\mc{O}_K$ with norm less than $x$. Generally, this theorem is given in the more effective form,
\beq
\pi_K(x) = \Li(x) + O_K\left(xe^{-c_K \sqrt{\log x}}\right),
\eeq
where $c_K>0$ is dependent on $K$.

 Our $B_K$ and other similarly-defined constants would allow us to produce analogues of Bertrand's postulate for the number field $K$, and though questions about the distribution of prime ideals are of great interest, it appears no attention has been paid to this problem outside of the case where $K= \mathbb{Q}$.
We would like to investigate $B_K$ for a non-trivial number field and the dependence it has on the invariants of the number field. 

As with the proof of the prime number theorem, the prime ideal theorem is obtained by finding a zero-free region of the Dedekind zeta function, $\zeta_K(s)$, which is defined for $\Re(s)>1$ as 
\beq \label{ded}
\zeta_K(s) :=\sum_{\substack{\mathfrak{a} \subseteq \mc{O}_K \\ \mf{a} \neq (0)}} \frac{1}{N(\mathfrak{a})^s}= \sum_{n=1}^\infty \frac{c_K(n)}{n^s}, 
\eeq
where $\mathfrak{a}$ are the ideals in $\mc{O}_K$, and has a meromorphic continuation to all $s \in \mathbb{C}$. Like the Riemann zeta function, $\zeta_K(s)$ also has a functional equation and has
only one pole at $s=1$, which is simple with residue $\rho_K$. This residue is related to the invariants of $K$ by the formula
\beq
\rho_K = \frac{2^{r_1}(2\pi)^{r_2} h_K R_K}{w_K \sqrt{|\Delta_K|}}.
\eeq
Here $\Delta_K$ is the discriminant of $K$,  $d=[K:\mathbb{Q}]$ is its degree, $h_K$ is its class number, $R_K$ is its regulator, $w_K$ is the number of roots of unity contained in $K$, and $r_1$ and $r_2$ are the number of real and complex embeddings of $K$, respectively. 

In 1977, Lagarias and Odlyzko \cite{LO} were able to state effective versions of the Chebotarev density theorem, which generalizes the prime ideal theorem to prime ideals whose Frobenius automorphisms lie in fixed conjugacy classes. This in turn specializes to an effective version of the prime ideal theorem, which we formulate here:
\begin{theorem}[Lagarias, Odlyzko \cite{LO}] \label{ell1}
If $K$ is a number field, there exist effectively computable positive constants $c_1$ and $c_2$, independent of $K$, such that if $x \geq \exp(10d(\log|\Delta_K|)^2)$ then
\beq \label{ello}
|\pi_K(x)-\Li(x)+(-1)^{\ep_K} \Li(x^\beta)| \leq c_1x \exp\left(-c_2 \sqrt{\frac{\log x}{d}}\right),
\eeq
where $\Li(x^\beta)$ only occurs if there exists an exceptional real simple zero, $\beta$, of $\zeta_K(s)$ such that $1-(4\log |\Delta_K|)^{-1}<\beta<1$. Also $\ep_K=0$ or $1$, depending on $K$. 

\vspace{2 mm}

\noindent If the Generalized Riemann Hypothesis (GRH) holds for $\zeta_K(s)$ then there exists an effectively computable positive absolute constant $c_3$ such that for $x>2$,
\beq\label{GRH}
|\pi_K(x)-\Li(x)| \leq c_3 \left(x^{\hf}\log(|\Delta_K|x^d) \right).
\eeq
\end{theorem}

By using the above estimates, we can make an effort to determine when $\pi_K(Ax)-\pi_K(x)>0$. The possible exceptional zero complicates what would otherwise be a fairly straightforward computation, and so we make use of an upper bound due to Stark \cite{stark}, which itself depends on whether or not $K$ is a normal field extension. The proof of the following theorem can be found in Section \ref{pidb}.


%

\begin{theorem} \label{goodcor}
Let $K\neq \mathbb{Q}$ be a finite field extension of $\mathbb{Q}$ such that there exists a tower of fields $\mathbb{Q} =K_0 \subset K_1 \subset K_2 \subset \cdots \subset K_m = K$ where each $K_i$ is a finite normal extension of $K_{i-1}$. For any $A>1$ there exists $c_A>0$, dependent only on $A$, such that for  $x > \exp(c_Ad(\log|\Delta_K|)^2)$, there is a prime ideal $\mf{p}$ in $\mc{O}_K$ with $N(\mf{p}) \in [x,A x]$. 

\vspace{2 mm}

\noindent Now suppose that $K\neq \mathbb{Q}$ be a finite, but the tower of normal field extensions does not exist. Let $\log |\Delta_K| \gg d (\log d)^\alpha$ for $\alpha \in [0,1]$. For any $A>1$ there exists $c_A>0$, dependent only on $A$, such that for  $x > \exp(c_Ad(\log d)^{2-2\alpha}(\log|\Delta_K|)^2)$, there is a prime ideal $\mf{p}$ in $\mc{O}_K$ with $N(\mf{p}) \in [x,A x]$. 

\vspace{2 mm}

\noindent Now only suppose that $K \neq \mathbb{Q}$ is a number field. If the GRH holds then for any $A>0$ there exists $c_A$, dependent only on $A$, such that for 
\[
x>c_A (\log|\Delta_K|+d)^2\log^4( \log |\Delta_K|+ d)
\]
 there is a prime ideal $\mf{p}$ in $\mc{O}_K$ with $N(\mf{p}) \in [x,Ax]$. 
\end{theorem}

\begin{remark}
We note from Minkowski's bound and Stirling's approximation that it is always the case that $\log |\Delta_K| \gg d$. So by specifying that $\log |\Delta_K| \gg d (\log d)^\alpha$ for $\alpha \in [0,1]$ we are not excluding any cases. 
\end{remark}

\noindent Since increasing $A$ means that we can decrease $c_A$, we can eventually get the following corollary extending Bertrand's postulate to number fields. 

\begin{cor} \label{goodcor2} There exists an absolute constant $c$ such that for any number field, $K\neq \mathbb{Q}$,  we have that 
\begin{itemize}
\item[a) ] $B_K \leq \exp(cd(\log |\Delta_K|)^2)$ if there exists a tower of fields $\mathbb{Q} =K_0 \subset K_1 \subset K_2 \subset \cdots \subset K_m = K$ where each $K_i$ is a finite normal extension of $K_{i-1}$.
\vspace{2mm}
\item[b) ] $B_K \leq \exp(cd(\log d)^{2-2\alpha} (\log |\Delta_K|)^2)$  if the tower of fields does not exist and when $\log |\Delta_K| \gg d (\log d)^\alpha$ for $\alpha \in [0,1]$. 
\vspace{2mm}
\item[c) ] $B_K\leq c  (\log|\Delta_K|+d)^2\log^4( \log |\Delta_K|+ d)$ if the GRH holds,
\end{itemize}

%
\end{cor}

The proof of the effective prime ideal theorem is quite technically involved. Since some of our interest in Bertrand's postulate is due to the brevity and elegance of the proofs for it, one would hope that in generalizing Bertrand's postulate to number fields we could obtain a comparable result using a less elaborate argument and without requiring information about the zeros of $\zeta_K(s)$. 

Let
\beq \label{eff1}
f_1(x,K) :=\left( \sum_{n \leq x} c_K(n)\right) - \rho_K x,
\eeq
where, as above, $c_K(n)$ are the coefficients of $\zeta_K(s)$ and $\rho_K$ is the residue of $\zeta_K(s)$ at $s=1$. That is, $f_1(x,K)$ is the error term for the number of ideals in $\mc{O}_K$ with norm less than $x$.
Information about $f_1(x,K)$ alone is enough information to obtain a generalized Bertrand's postulate for a finite field extension, $K$ of $\mathbb{Q}$. We prove the following result in Section \ref{bpc}.

\begin{theorem} \label{bigt} Let $K$ be a number field, as above.
Suppose for fixed $0<\alpha<1$, there exists some $\mc{C}_K>0$, determined by the invariants of $K$, such that
\beq \label{generaleff1}
|f_1(x,K)| \leq \mc{C}_Kx^{\alpha} 
\eeq
for all $x\geq 1$. Then for any $x>1,$ a prime ideal, $\mathfrak{p}$, exists in $\mathcal{O}_K$ such that $N(\mathfrak{p}) \in [x,Ax]$, whenever 
\beq \label{abound}
 \log A \gg \frac{\mc{C}_K}{\rho_K}\left(\frac{d+2}{1-\alpha}\right)+d.
\eeq
\end{theorem}

\begin{remark}
It is common notation that $f(x) \ll g(x)$ indicates that $|f(x)| \leq Cg(x)$ for particular values of $x$ for some independent constant $C$. Throughout this work, however, we will say $f \ll g$ if  there exists some constant $C$, which can be chosen independently of any of the variables or invariants that may define $f$ and $g$, such that $|f| \leq C|g|$, unless otherwise specified. We say that $f \ll_d g$ if $C$ has some dependence on $d$, and similarly for other variables. Our big-O notation reflects this as well. For example, we say $f = g +O_d(x)$ if $(f-g) \ll_d x$.  
\end{remark}

In 1972, by effectivizing Landau's theorem, Sunley \cite{Sunley} was able to derive a completely effective bound for $|f_1(x,K)|$. 

\begin{theorem}[Sunley \cite{Sunley}] For $f_1(x,K)$ as in \eqref{eff1}, we have that
\beq
|f_1(x,K)| < e^{56d+5}(d+1)^{5(d+1)/2}|\Delta_K|^{1/(d+1)}\left( \log^d |\Delta_K| \right)x^{(d-1)/(d+1)}. 
\eeq
\end{theorem}
Combining this with Theorem \ref{bigt}, and employing a theorem due to Friedman \cite{F} which gives us that
\beq
\frac{R_K}{w_K} \geq \frac{9}{100},
\eeq
for all number fields $K$, we are able to produce the following corollary as a proof of concept.
\begin{cor}\label{sun} Let $K$ be a number field with Bertrand constant $B_K$, then
\beq \log B_K \ll \frac{e^{\frac{5}{2}(d+5)\log(d+1)+(56-\log 2)d+5}}{h_K}|\Delta_K|^{\frac{1}{2}+\frac{1}{d+1}} \log^d |\Delta_K|+d 
\eeq
where the implied constant is absolute. 
\end{cor}
This result is significantly worse than Corollary~\ref{goodcor2}, but if it has an advantage it is that the absolute constant is significantly easier to compute from the proof of Theorem~\ref{bigt}. Indeed, we can do better than Sunley if we restrict our attention to just growth in the $|\Delta_K|$ aspect rather than attempting a hybrid bound in all the invariants of $K$.

With this in mind, we consider the following proposition.
\begin{prop} \label{coreprop} For all $x\geq 1$ and small $\delta >0 $ such that $\frac{1}{3d} > \delta$ we have that
\begin{align} \label{eff1bound}
& \sum_{n \leq x} c_K(n) =\rho_K x+O_{d,\delta}\left( (\zeta_K(1+\tfrac{\delta}{2})|\Delta_K|^\delta +1)x^{1-\frac{\delta}{2}}   \right).
\end{align}
\end{prop}
\noindent This asymptotic is obtained almost directly from the work of Kuo and Murty \cite{KM}, albeit in such a way that the contribution from the discriminant is mitigated at the expense of growth in $x$. The proof of it can be found in the appendix of this paper. While we know that $\zeta_K(1+\frac{\delta}{2}) \leq \zeta^d(1+\frac{\delta}{2})$ and so $\zeta_K(1+\frac{\delta}{2})$ can be absorbed into the implied constant, we also know that 
\beq
\lim_{\delta \to 0} \frac{\delta}{2} \zeta_K
\left(1+\frac{\delta}{2}\right) \to \rho_K,
\eeq
 though it is not obvious how small $\delta$ needs to be relative to the invariants of $K$ for this to be a good approximation. Still, we heuristically expect $\zeta_K(1+\frac{\delta}{2})$ to cancel the $\rho_K$ in \eqref{abound} in exchange for a large contribution due to $\delta$, so we keep track of it. The obstacle to understanding the bound in the $\delta$ and $d$ aspects is finding a good effective estimate for the implied constants for the bound $c_K(n) \ll_{\delta,d} n^\delta$. 

Inputting \eqref{eff1bound} into \eqref{abound}, and again employing a theorem due to Friedman \cite{F}, we are able to produce the following corollary.

\begin{cor}\label{maint}
Let $K$ be a number field with Bertrand Constant $B_K$, then
\beq \label{poor} 
B_K \leq \exp\left(M_{d,\delta}\left(\frac{\zeta_K(1+\frac{\delta}{2})|\Delta_K|^\delta}{\rho_K}+ \frac{|\Delta_K|^{\frac{1}{2}}}{h_K} \right)\right),
\eeq
for some constant $M_{\delta,d}>0$ dependent on $d$ and $\delta$ where $\frac{1}{3d} > \delta >0$. 
\end{cor}

The size of $ \frac{|\Delta_K|^{\frac{1}{2}}}{h_K}$ depends on the existence of a Siegel zero, but in the case of totally complex number fields is heuristically likely to grow like $\log |\Delta_K|$.  If we could indeed let $\zeta_K(1+\delta)/\rho_K =O_\delta(1)$ then we would be left with the $|\Delta_K|^\delta$ term as a main term. So if we hope to match the results in Theorem \ref{maint} in the $\Delta_K$ aspect in any case, we would need to let $\delta = 2\log(\log |\Delta_K|)/\log(|\Delta_K|)$, but then we would be limited by the lack of effectiveness in the $\delta$ aspect. 

Though the results of Theorem~\ref{bigt} and our current bounds for $|f_1(x,K)|$ are apparently worse than those that can be obtained from careful analysis of the effective prime ideal theorem, one would hope that they might be put to better use in certain special cases, such as for quadratic fields. This has some overlap with the older problem of Bertrand's postulate for primes in arithmetic progressions, and may be an avenue for further research. A thorough treatment for this topic can be found by Moree \cite{Moree}.

\section{Bertrand's Postulate from the Prime Ideal Theorem}\label{pidb}

In this section we prove Theorem \ref{goodcor}. 

\begin{proof}[Proof of Theorem \ref{goodcor}] 
First suppose that no exceptional zero exists. By Theorem \ref{ell1} we only have to show that for any $A>1$ there exists $c_A$, independent of $|\Delta_K|$ and $d$, such that
\beq
\pi(Ax)-\pi(x)  >\Li(Ax)-\Li(x) -2c_1 Ax \exp\left(-c_2\sqrt{\frac{\log x}{d}}\right)>0
\eeq
for $x \geq \exp(c_A d(\log |\Delta_K|)^2)$. This is also the case if an exceptional zero exists and $\epsilon_K=1$ since $\Li(x)$ is an increasing function for $x>2$. From partial integration we have that
\beq
\Li(Ax)-\Li(x) = \frac{Ax}{\log Ax}-\frac{x}{\log x}+ \int_{x}^{Ax} \frac{dt}{(\log t)^2},
\eeq
so it suffices to show that
\beq \label{errorterm}
A\left(\frac{ \log x}{\log A x}\right) >1+ 2A c_1 (\log x) \exp\left(-c_2 \sqrt{\frac{\log x}{d}}\right).
\eeq
Since only $A$ determines how large $x$ needs to be for $A\log x / \log Ax$ to be larger than one, we need only show that we can make the exponential term sufficiently small by controlling the size of $c_A$ independently of $d$ and $\Delta_K$. Indeed, let $x = \exp(c_A d(\log |\Delta_K|)^2)$, then we see that,
\beq
 2Ac_1(\log x) \exp\left(-c_2 \sqrt{\frac{\log x}{d}}\right) = 2Ac_1 (c_A d (\log |\Delta_K|)^2)|\Delta_K|^{-c_2\sqrt{c_A}}. \label{ez1}
\eeq   
 Taking advantage of Minkowski's bound and Stirling's approximation we can say that 
\begin{equation} 
d|\Delta_K|^{-1} \leq \left( \frac{4}{\pi}\right)^{d} \frac{(d!)^2}{d^{2d-1}} \sim 2\pi \left(\frac{4}{\pi e^2} \right)^d, \label{mink}
\end{equation} 
as $d$ gets large. Thus the exponential term in \eqref{ez1} will decrease as $c_A$ increases past a point that can be chosen independently of $K$, and so we can say that for $x > \exp(c_A d(\log |\Delta_K|)^2)$, 
\beq
(\log x) \exp\left(-c_2 \sqrt{\frac{\log x}{d}}\right) \leq  (c_A d (\log |\Delta_K|)^2)|\Delta_K|^{-c_2\sqrt{c_A}}. \label{ez11}
\eeq   
We see the upper bound can be made uniform in $K$ and also vanishes as $c_A \to \infty$, giving the proposition in this case.

 For the case of the General Riemann Hypothesis, we follow the same reasoning but replace \eqref{GRH} with \eqref{ello}. So when 
 \[
 x>c_A (\log|\Delta_K|+d)^2\log^4( \log |\Delta_K|+ d),
 \]
   it suffices to observe, via substitution, that the term
 \begin{align}
\left(\frac{\log x}{x} \right) & x^{\frac{1}{2}} \log(|\Delta_K| x^d) 
 \end{align}
 can be made arbitrarily small by increasing $c_A$ independently of $K$.

Now suppose an exceptional zero exists with $\beta > 1-(4 \log |\Delta|_K)^{-1}$ and $\epsilon_K =0$. Then we need to show that
\beq
\Li(Ax)-\Li(x)-\left(\Li((Ax)^\beta)-\Li(x^\beta)\right) -2c_1 Ax \exp\left(-c_2\sqrt{\frac{\log x}{d}}\right)>0,
\label{partial3} \eeq
for $x > \exp(c_A d(\log |\Delta_K|)^2)$. We have from Stark \cite{stark} that, if $K$ is a normal extension, we can assume 
\beq \label{starkrange}
1-(4 \log |\Delta|_K)^{-1}<\beta<1-c_4|\Delta_K|^{-1/d},
\eeq
if $\beta$ exists, for some effectively computable positive constant $c_4$. Differentiation shows us that $\Li((Ax)^\beta)-\Li(x^\beta)$ increases as $\beta$ increases, so we can just let $\beta = 1-c_4|\Delta_K|^{-1/d}$.

Thus we can deduce from \eqref{partial3} and partial integration that, for $A>1$ and $x> \exp(10d(\log |\Delta_K|)^2)$, we have $\pi_K(Ax)>\pi_K(x)$ if 
\begin{align} \label{effective1}
\frac{\beta Ax-(Ax)^\beta}{\beta \log Ax} + \int_{(Ax)^\beta}^{Ax} \frac{dt}{(\log t)^2} >& \frac{\beta x-x^\beta}{\beta \log x}  + \int_{x^\beta}^x \frac{dt}{(\log t)^2}  \\
&\notag + 2c_1Ax \exp\left(-c_2 \sqrt{\frac{\log x}{d}}\right).
\end{align}
Taking derivatives, for fixed $\beta$ we see that $\int_{x^\beta}^x \frac{dt}{(\log t)^2}$ is increasing in $x$ for $x>4$ when $\beta>\frac{1}{2}$. Thus we can drop the integral terms from both sides of \eqref{effective1} to get an inequality that still yields $\pi_K(Ax)>\pi_K(x)$. Rewriting this, we get

\begin{align}
\left(\frac{\beta-(Ax)^{\beta-1}}{\beta -x^{\beta-1}} \right) \frac{\log x}{\log Ax}>& \frac{1}{A}+\frac{2\beta c_1 \log x}{\beta - x^{\beta-1}}\exp\left(-c_2\sqrt{\frac{\log x}{d}}\right).
\end{align}
We see the left hand side is bigger than $\frac{\log x}{\log Ax}$ for any $A>1$, and furthermore $\frac{\log x}{\log Ax} - \frac{1}{A}>0$ as $x$ gets large. Thus again we just need to show that there exists $c_A$, independent of $d$ and $|\Delta_K|$, such that
 when $x>\exp(c_Ad (\log |\Delta_K|)^2)$ we can make the remaining exponential term arbitrarily small.  Let $c_A>10$ and $c_2\sqrt{c_A}>\frac{5}{2}$. If $x>\exp(c_A d (\log |\Delta_K|)^2)$, we can choose $c_A$ to be large enough independently of $d$ and $|\Delta_K|$ such that , 
\begin{align} \label{bound1}
\frac{2\beta c_1 \log x}{\beta - x^{\beta-1}}\exp\left(-c_2\sqrt{\frac{\log x}{d}}\right)  
 < \frac{2\beta c_1 c_Ad(\log |\Delta_K|)^2|\Delta_K|^{-c_2\sqrt{c_A}} }{\beta-\exp(c_A d(\log \Delta_K|)^2(\beta-1))}.
\end{align}
Letting $\beta =1-c_4 |\Delta_K|^{-\frac{1}{d}}$, we have that
\begin{align}
\frac{2\beta c_1 \log x}{\beta - x^{\beta-1}} e^{-c_2\sqrt{\frac{\log x}{d}}} <
 \left.\frac{2 c_1 c_Ad(\log |\Delta_K|)^2|\Delta_K|^{\frac{1}{d}-c_2\sqrt{c_A}} }{ |\Delta_K|^{1/d}(1-\exp(-c_4c_A d (\log |\Delta_K|)^2|\Delta_K|^{-1/d}))-c_4}\right. . \label{uniform}
\end{align}
It is not difficult to see that the denominator of this upper bound is bounded below and positive for sufficiently large $c_A$, independent of $|\Delta_K|$ and $d$. So in this case we get that
\beq
 \frac{2\beta c_1 \log x}{\beta - x^{\beta-1}} e^{-c_2\sqrt{\frac{\log x}{d}}} \ll c_A d(\log |\Delta_K|)^2|\Delta_K|^{\frac{1}{d}-c_2\sqrt{c_A}}, \label{uniform2}
\eeq
where the implied constant can be made independent of $c_A$, $d$ and $|\Delta_K|$ provided that $c_A$ is sufficiently large. 
 From \eqref{mink} we see that we can make \eqref{uniform2} independent of $|\Delta_K|$ and $d$ provided $c_2\sqrt{c_A}>\frac{5}{2}$, and further we see that this bound goes to zero as $c_A \to \infty$. 

 If the extension is not normal but there exists a tower of fields $\mathbb{Q} \subset K_1 \subset K_2 \subset \cdots \subset K_m = K$, such that $K_{i}$ is normal over $K_{i-1}$, then Stark's bound allows for the possibility that
 \beq
   1-c_4|\Delta_K|^{-1/d} < \beta < 1-(16 \log|\Delta_K|)^{-1}
 \eeq
  if $16c_4 \log|\Delta_K| > |\Delta_K|^{1/d}$. When this occurs,  \eqref{uniform} becomes
 \beq \label{lastineq}
 \frac{2\beta c_1 \log x}{\beta - x^{\beta-1}} e^{-c_2\sqrt{\frac{\log x}{d}}}  < \left.\frac{32 c_1 c_Ad(\log |\Delta_K|)^3|\Delta_K|^{-c_2\sqrt{c_A}} }{ 16(\log|\Delta_K|)  (1-\exp(-\frac{1}{16}c_Ad (\log |\Delta_K|)))-1},\right. 
 \eeq
 and again we see that this bound uniformly goes to zero as $c_A \to \infty$.
 
 Finally, if the extension is not normal, nor does there exist a tower of field extensions as above, then Stark's bound allows for the possibility that
 \beq
 1-c_4|\Delta_K|^{-1/d} < \beta < 1-(4d! \log|\Delta_K|)^{-1},
 \eeq
 if $4c_4 d! \log|\Delta_K| > |\Delta_K|^{1/d}$. When this occurs,  \eqref{uniform} becomes
 \beq \label{lastineq}
 \frac{2\beta c_1 \log x}{\beta - x^{\beta-1}} e^{-c_2\sqrt{\frac{\log x}{d}}}  < \left.\frac{8 c_1 c_Ad(\log |\Delta_K|)^3|\Delta_K|^{-c_2\sqrt{c_A}} d! }{ 4d!(\log|\Delta_K|)  (1-\exp(-\frac{1}{4}c_Ad (\log |\Delta_K|)/d!))-1}\right. . 
 \eeq
 This bound still decays as $c_A \to \infty$, and is uniform in $|\Delta_K|$ but not necessarily in $d$. We see from \eqref{mink} and Stirling's approximation that if we take $x>\exp(c_A d (\log d)^2 (\log |\Delta_K|)^2)$ instead, effectively replacing each occurrence of $c_A$ in \eqref{lastineq} with $c_A (\log d)^2$, this bound can be made uniform in $d$. 
Indeed, it is enough that we take $x>\exp(c_A d(\log d)^{2-2\alpha} (\log |\Delta_K|)^2)$ so long as $\log |\Delta_K| \gg d(\log d)^\alpha$.

\end{proof}

\section{Bertrand's Postulate from Counting Ideals} \label{bpc}
In this section we prove Theorem \ref{bigt}.

\begin{proof}[Proof of Theorem \ref{bigt}.] For an ideal $\mathfrak{a}$ in $\mathcal{O}_K$, let $\Lambda_K(\mathfrak{a}):=\log N(\mathfrak{p}) $ when $\mathfrak{a}=\mathfrak{p}^k$, where $\mathfrak{p}$ henceforth denotes a prime ideal in $\mathcal{O}_K$ above a prime ideal $(p) \subset \mathbb{Z}$, and zero otherwise. This is the natural extension of the von Mangoldt function to $K$, where
\beq
-\frac{\zeta_K'(s)}{\zeta_K(s)} = \sum_{\mathfrak{a} \subseteq \mathcal{O}_K} \frac{\Lambda_K(\mathfrak{a})}{N(\mathfrak{a})^s} = \sum_{n=1}^\infty \frac{\Lambda^\#_K(n)}{n^s}.
\eeq
We similarly define the Chebyshev function for $K$,
\beq
\psi_K(x): = \sum_{N(\mathfrak{a})\leq x} \Lambda_K(\mathfrak{a}) = \sum_{n \leq x} \Lambda_K^\# (n). 
\eeq

\begin{lemma}For $x \geq 1$ 
\beq \label{lemma1}
\sum_{n \leq x} \frac{\Lambda^\#_K(n)}{n} = 
\sum_{N(\mathfrak{p}) \leq x} \frac{\log N(\mathfrak{p})}{N(\mathfrak{p})}
+ O(d),
\eeq
where the above right-hand sum is over norms of prime ideals, $q$. \\

\vspace{2 mm}

\noindent For $1>\alpha>0$,
\beq \label{lemma2}
\sum_{n \leq x} \frac{\Lambda_K^\#(n)}{n^\alpha} \ll d\frac{(x^{1-\alpha}-\alpha)}{1-\alpha}.
\eeq
\end{lemma}

\begin{proof}[Proof of lemma] Let 
\beq 
\phi(x) := \sum_{n \leq x} \frac{\Lambda^\#_K(n)}{n}
-\sum_{N(\mathfrak{p}) \leq x} \frac{\log N(\mathfrak{p})}{N(\mathfrak{p})}.
\eeq
We see that
\begin{align}
|\phi(x)| \leq & \sum_{N(\mathfrak{p}) \leq x} \sum_{n=2}^\infty \frac{\log N(\mathfrak{p})}{N(\mathfrak{p})^n} \leq \sum_{p \leq x} \sum_{\mathfrak{p} \, \cap \, \mathbb{Z} =(p)}  \sum_{n=2}^\infty  \frac{\log N(\mathfrak{p})}{p^n} \leq \sum_{p \leq x} \sum_{n=2}^\infty \frac{d \log p }{p^n} \notag \\
& \leq d \sum_{p \leq x} \frac{\log p}{p^2-p} \leq d \sum_{m=2}^\infty \frac{2\log m}{m^2} = O(d), 
\end{align}
which gives \eqref{lemma1}. 
To get \eqref{lemma2} we can use Abel's partial summation formula to get that
\beq \label{abel2}
\sum_{n \leq x} \frac{\Lambda^\#_K(n)}{n^\alpha} = \psi_K(x)x^{-\alpha} + \alpha \int_1^x \psi_K(u)u^{-\alpha-1} \ du.
\eeq
Since $\Lambda^\#_K(n) \leq d\Lambda(n)$, where $\Lambda(n):=\Lambda_{\mathbb{Q}}(n)$ is the classical von Mangoldt function, we have that $\psi_K(x) \leq d \psi(x)$
for all $x\geq 1$, where $\psi(x) :=\psi_{\mathbb{Q}}(x)$ is the classical Chebyshev function.  It is easily shown that $\psi(x) \ll x$, so $\psi_K(x) \ll dx$. Putting this into \eqref{abel2} we get that 
\beq 
\sum_{n \leq x} \frac{\Lambda_K^\#(n)}{n^\alpha} \ll d\left( x^{1-\alpha}+\frac{\alpha}{1-\alpha}(x^{1-\alpha}-1)\right)
\eeq
which is a restatement of \eqref{lemma2}.
\end{proof}
We see that, by the unique prime factorization of ideals in the ring of integers of a number field, we have that
\begin{align} \label{bigidea}
\sum_{n \leq x} c_K(n)\log n &= \sum_{N(\mathfrak{a})\leq x} \log N(\mathfrak{a}) = \sum_{N(\mathfrak{a})\leq x} \sum_{\mathfrak{b} | \mathfrak{a}} \Lambda_K(\mathfrak{b}) \notag \\
& = \sum_{N(\mathfrak{b}) \leq x} \Lambda_K(\mathfrak{b}) \sum_{N(\mathfrak{a})\leq \frac{x}{N(\mathfrak{b})}} 1 
= \sum_{n \leq x} \Lambda_K^\#(n) \sum_{m \leq \frac{x}{n}} c_K(m). 
\end{align}
Letting 
\beq \label{eff2}
f_2(x,K) := \left(\sum_{n \leq x} c_K(n)\log n \right) - \rho_K( x\log x -x+1), 
\eeq
we get that \eqref{bigidea} becomes
\beq
\rho_K(x \log x -x+1)+f_2(x,K) = \sum_{n \leq x} \Lambda_K^\#(n) \left(\rho_K\tfrac{x}{n} + f_1(\tfrac{x}{n},K) \right).
\eeq
Thus by \eqref{lemma1} we have
\beq \label{ceecue}
\sum_{N(\mathfrak{p}) \leq x} \frac{\log N(\mathfrak{p})}{N(\mathfrak{p})}= \log x+\frac{1}{x\rho_K}\left( f_2(x,K) -\sum_{n \leq x} \Lambda_K^\#(n)f_1(\tfrac{x}{n},K) \right) +O(d).
\eeq
Now by combining \eqref{generaleff1} and \eqref{lemma2} we have that 
\begin{align} \label{term1}
&\sum_{n\leq x}  \Lambda_K^\#(n)f_1(\tfrac{x}{n},K) \ll d\mc{C}_K\frac{x-\alpha x^{\alpha}}{1-\alpha}.
\end{align}

To bound $f_2(x,K)$ we make use of the following lemma.
\begin{lemma}
If for $0<\alpha<1$ and $x\geq 1$,
\beq 
|f_1(x,K)| \leq \mc{C}_Kx^{\alpha} 
\eeq
for some $\mc{C}_K>0$, then
 \beq \label{eff2bound}
|f_2(x,K)| \leq \mc{C}_Kx^\alpha \left(\log x+\frac{1-x^{-\alpha}}{\alpha}\right).
 \eeq
\end{lemma}
\begin{proof}[Proof of lemma.]
This follows from the bound in \eqref{generaleff1} by another application of Abel's summation formula. Indeed, 
\beq
\sum_{n \leq x} c_K(n) \log n = \left(\sum_{n \leq x} c_K(n)\right)\log x - \int_1^x \left(\sum_{n \leq u}  c_K(n)\right) \frac{du}{u}
\eeq
so we have
\begin{align}
\sum_{n \leq x} c_K(n) \log n &= \left(\rho_Kx+f_1(x,K)\right)\log x  - \int_1^x \left(\rho_Ku+f_1(u,K) \right) \frac{du}{u} \notag \\
&  = \rho_K(x\log x -x+1)+f_1(x,K)\log x - \int_1^x f_1(u,K) \frac{du}{u}
\end{align}
and \eqref{eff2bound} follows. 
\end{proof}

Substituting \eqref{term1} and \eqref{eff2bound} back into \eqref{ceecue} we get
\begin{align}
&\sum_{N(\mathfrak{p}) \leq x} \frac{\log N(\mathfrak{p})}{N(\mathfrak{p})}  \\
&=\log x + O\left(\frac{\mc{C}_K}{\rho_K}\left(\frac{d(1-\alpha x^{\alpha-1})}{1-\alpha}+x^{\alpha-1}\log x +\frac{x^{\alpha-1}-x^{-1}}{\alpha}\right)+d\right). \notag
\end{align}
And so finally, for any $A\geq 1$, 
\beq
\sum_{N(\mathfrak{p}) \leq Ax} \frac{\log N(\mathfrak{p})}{N(\mathfrak{p})} -\sum_{N(\mathfrak{p}) \leq x} \frac{\log N(\mathfrak{p})}{N(\mathfrak{p})} = \log A +O\left(\frac{\mc{C}_K}{\rho_K}\left(\frac{d+2}{1-\alpha}\right)+d\right).
\eeq
Thus, for any $x>1$, a prime ideal, $\mathfrak{p}$, exists in $\mathcal{O}_K$ such that $N(\mathfrak{p}) \in [x,Ax]$, so long as 
\beq 
 \log A \gg \frac{\mc{C}_K}{\rho_K}\left(\frac{d+2}{1-\alpha}\right)+d,
\eeq
which proves Theorem \ref{bigt}.

\end{proof}

\appendix
\section{Ideal Counting}\label{app}

The proof of Proposition \ref{coreprop} proceeds with only subtle variation from the proof of the main theorem due to Kuo and Murty in their relevant work \cite{KM}. Changes are made to account for the presence of the simple pole at $s=1$ and that our goal is to minimize growth in $|\Delta_K|$ aspect, possibly at the expense of growth in the $x$ aspect.

First we use the following result due to Rademacher \cite{R}, arising from the sharper version of the Phragm\'{e}n-Lindel\"{o}f theorem to obtain the convexity bound for $\zeta_K(s)$ in the critical strip. 
\begin{theorem}[Rademacher \cite{R}]\label{lem:con}
For $\sigma,\eta,t \in \mathbb{R}$ such that  $\frac{3}{2}\geq \sigma>1$ and $1-\sigma<\eta <\sigma$, we have that
\beq
\zeta_K(\eta+it) \leq  3 \left(\left|\Delta_K\right|\left(\frac{|1+\eta+it|}{2\pi}\right)^{d}\right)^{\frac{\sigma-\eta}{2}}\frac{|1+\eta+it|}{|\eta-1+it|}\zeta_K(\sigma).
\eeq
\end{theorem}

\noindent We are now ready to prove the proposition. 

\begin{proof}[Proof of Proposition \ref{coreprop}.]

Henceforth we will use the following notation:
\beq
\int_{(c,T)} f(s) \ ds:= \frac{1}{2\pi i} \int_{c-iT}^{c+iT} 
 f(s) \ ds.
\eeq
Let $T\geq 1$, $2>c>1$ and, until stated otherwise, let $x \in \mathbb{N}+\frac{1}{2}$. By Perron's formula, we have that
\beq
\sum_{n \leq x} c_K(n) = \int_{(c,T)} \zeta_K(s) \frac{x^s}{s} \ ds + O \left(\sum_{n=1}^\infty \left( \frac{x}{n}\right)^c c_K(n) \min\left(1,\frac{1}{T|\log \frac{x}{n}|} \right)  \right). 
\eeq
Thus by Cauchy's residue theorem, we have that for $0<\eta<1$,
\beq
\sum_{n \leq x} c_K(n) =\rho_Kx +E_1+E_2+E_3
\eeq 
where 
\begin{align}
&E_1:=\int_{(\eta,T)} \zeta_K(s) \frac{x^s}{s} \ ds \\
&E_2: = \frac{1}{2\pi i} \left( \int_{\eta+iT}^{c+iT} \zeta_K(s) \frac{x^s}{s} \ ds -\int_{\eta-iT}^{c-iT} \zeta_K(s) \frac{x^s}{s} \ ds  \right) 
\\
&E_3:= O \left(\sum_{n=1}^\infty \left( \frac{x}{n}\right)^c c_K(n) \min\left(1,\frac{1}{T|\log \frac{x}{n}|} \right)  \right).
\end{align}
We can now use Theorem \ref{lem:con} to bound $E_1+E_2+E_3$. Since $(1+\frac{t^2}{\eta^2}) \gg (1+|t|)^2$, we have that
\begin{align} \label{e1bound}
 |E_1|&\ll \left| \int_{(\eta,T)} \zeta_K(s) \frac{x^s}{s} \ ds \right| \\
& \ll_d \left| \Delta_K \right|^{\frac{c-\eta}{2}} \zeta_K(c)x^\eta \int_{-T}^T \frac{(1+|t|)^{d(\frac{c-\eta}{2})+1}}{|\eta-1+it||\eta+it|} \ dt \notag \\
& \ll_d \frac{1}{\eta(1-\eta)}\left| \Delta_K \right|^{\frac{c-\eta}{2}} \zeta_K(c)x^\eta \int_{0}^T (1+t)^{d(\frac{c-\eta}{2})-1} \ dt \notag \\
& \ll_d \frac{1}{\eta(1-\eta)^2}\left| \Delta_K \right|^{\frac{c-\eta}{2}} \zeta_K(c)x^\eta (1+T)^{d(\frac{c-\eta}{2})}. \notag \end{align}
Similarly,
\begin{align}
& |E_2| \ll \left( \left| \int_{\eta+iT}^{c+iT} \zeta_K(s) \frac{x^s}{s} ds\right| +\left| \int_{\eta-iT}^{c-iT} \zeta_K(s) \frac{x^s}{s} ds\right|  \right) \\
& \ll_d \zeta_K(c) (1+T) \int_{\eta}^c \left(\left|\Delta_K\right|(1+T)^{d}\right)^{\frac{c-\lambda}{2}}x^\lambda|\lambda+iT|^{-1}|1-\lambda+iT|^{-1} \ d\lambda. \notag
\end{align}
Since for $T\geq 1$ we have that $|\lambda+iT||1-\lambda+iT| \gg (1+T)^2$, it follows that
\begin{align}
|E_2| & \ll_d  \zeta_K(c)\left|\Delta_K\right|^{\frac{c}{2}} (1+T)^{\frac{dc}{2}-1} \int_{\eta}^c \left(\frac{x}{\sqrt{\left|\Delta_K\right|}(1+T)^{\frac{d}{2}}}\right)^\lambda \ d\lambda \\
& \notag  \ll_d \zeta_K(c)\left|\Delta_K\right|^{\frac{c}{2}} (1+T)^{\frac{dc}{2}-1} \left|\log \left(\frac{x}{\sqrt{\left|\Delta_K\right|}(1+T)^{\frac{d}{2}}}\right) \right|^{-1} \\
& \notag \ \ \ \ \ \ \ \  \ \ \ \ \ \ \ \times \left| \left(\frac{x}{\sqrt{\left|\Delta_K\right|}(1+T)^{\frac{d}{2}}}\right)^c- \left(\frac{x}{\sqrt{\left|\Delta_K\right|}(1+T)^{\frac{d}{2}}}\right)^\eta \right|.
\end{align}
So when
\beq \label{morelog}
\left| \log\left(\frac{x}{\sqrt{\left|\Delta_K\right|}(1+T)^{\frac{d}{2}}}\right)\right| \geq \log(3/2)
\eeq
we have that
\begin{align} \label{e2bound}
|E_2| \ll_d \zeta_K(c)\left( \frac{x^c}{1+T}+x^{\eta} \left|\Delta_K\right|^{\frac{c-\eta}{2}} (1+T)^{\frac{d(c-\eta)}{2}-1} \right).
\end{align}
When
\beq \label{lesslog}
\left| \log\left(\frac{x}{\sqrt{\left|\Delta_K\right|}(1+T)^{\frac{d}{2}}}\right)\right| \leq \log(3/2),
\eeq
we note that $(y^c-y^\eta)/\log(y)$ is bounded uniformly for $y \in [\frac{2}{3},\frac{3}{2}]$ and all specified values of $c$ and $\eta$. So we can just say  $(y^c-y^\eta)/\log(y)\ll y^c$ in this range. 
Thus 
the bound on $E_2$ given in \eqref{e2bound} holds regardless of the relationship between $x$ and $\sqrt{\left|\Delta_K\right|}(1+T)^{\frac{d}{2}}$.

Finally, recalling that $x \in \mathbb{N}+\frac{1}{2}$ and that $|\log(1-x)| \gg |x|$ for $x \in [-1,\frac{1}{2})$, we have 
\begin{align} \label{multistep}
E_3: &= x^c \sum_{n=1}^\infty \frac{c_K(n)}{n^c} \min\left(1, \frac{1}{T|\log \frac{x}{n}|}\right) \leq x^c \sum_{n=1}^\infty \frac{c_K(n)}{n^c} \frac{1}{T|\log \frac{x}{n}|} \\
&\notag \ll \frac{x^c}{T} \zeta_K(c) +\sum_{\frac{x}{2} < n \leq 2x} \left(\frac{x}{n}\right)^c c_K(n) \frac{1}{T|\log \tfrac{x}{n}|} \\
&\notag \ll \frac{x^c}{T } \zeta_K(c) + \frac{1}{T} \sum_{\frac{x}{2} < n \leq 2x} c_K(n) \frac{n}{|x-n|} \notag \\
&\notag  \ll \frac{x^c}{T } \zeta_K(c) +  \frac{C_{d,\varepsilon}x^{1+\varepsilon}}{T}\sum_{\frac{x}{2} \leq n \leq 2x } \frac{1}{|x-n|}, 
\end{align}
for small $\varepsilon>0$, and  $C_{d,\varepsilon}$ is the constant such that
\beq \label{bb2}
c_K(n) \leq C_{d,\varepsilon} n^\varepsilon,
\eeq
for all $n \in \mathbb{N}$.

%

Now since
\beq
\sum_{\frac{x}{2} \leq n \leq 2x} \frac{1}{|n-x|}\leq2 \sum_{j=0}^{x-\frac{1}{2}} \frac{1}{j+\frac{1}{2}} \ll \log(1+x),
\eeq
if we let $\varepsilon=(c-1)/2$, then for $T\geq 1$ we have that \eqref{multistep} becomes
\begin{align} \label{e3bound}
|E_3|& \ll_\delta \frac{x^c}{1+T} \zeta_K(c)+ \frac{C^{(2)}_{d,\frac{c-1}{2}}x^{\frac{c+1}{2}}}{1+T}.
\end{align}
where the implicit constant $C_{d,\frac{c-1}{2}}$ is changed to $C^{(2)}_{d,\frac{c-1}{2}}$ to account for the implied constant in the bound $ \log(x) \ll_\vep x^\vep $. We remark that we do not bother measuring the contribution of the degree $d$, nor that of $\delta=c-1$, in our main theorem as  $C^{(2)}_{d,\frac{\delta}{2}}$ is likely much worse than reality. It is unclear at present how to remove the dependence on this term. 

 Combining \eqref{e1bound}, \eqref{e2bound}, and \eqref{e3bound} we get that
\beq \label{alles}
|E_1|+|E_2| + |E_3| \ll_d  \zeta_K(c) \left(\frac{x^\eta\left(\left| \Delta_K \right|(1+T)^{d}  \right)^{\frac{c-\eta}{2}}}{\eta(1-\eta)^2}+\frac{x^c}{1+T}\right) + \frac{C^{(2)}_{d,\frac{c-1}{2}}x^{\frac{c+1}{2}}}{1+T}.
\eeq
For some small $\delta >0$ such that $\frac{1}{3d} > \delta$, let $c=1+\delta$ and $\eta=1-\delta$, then let
\beq \label{tcon}
1+T =x^{\frac{2\delta}{1+d\delta}}.
\eeq
From this, \eqref{alles} becomes 
\begin{align} \label{main}
& \sum_{n \leq x} c_K(n) =\rho_K x+O_{d,\delta}\left((\zeta_K(1+\delta) |\Delta_K|^\delta+1) x^{1-\frac{\delta}{2}}    \right)
\end{align}
when $x\in \mathbb{N}+\frac{1}{2}$ and, since $T \geq 1$, we also have the constraint $x \geq Q$ due to \eqref{tcon}, where $Q$ is a constant dependent on $\delta$ and $d$. For $x<Q$ we can just let $T = 1$ in \eqref{alles}, and since $x^c \leq x^\eta Q^{2\delta}$ in this range we can say \eqref{main} holds for all $x \in \mathbb{N}+\frac{1}{2}$ and $x\geq 1$. 

For $x \notin \mathbb{N}+ \frac{1}{2}$ we can replace $x$ with $\lfloor x \rfloor + \frac{1}{2}$ in $\sum_{n \leq x} c_K(n)-\rho(x)$ and note the difference will be on the order of $\rho_K$ at most. Since $\zeta_K(1+\delta) \sim \frac{\rho_K}{\delta}$ we can replace $\zeta_K(1+\delta)$ with $\zeta_K(1+\frac{\delta}{2})$ to supersede that $\rho_K$ term, and thus get \eqref{main} for all $x$ and complete the proof of the proposition.
%
\end{proof}

\bibliography{bertrandbib}	
\bibliographystyle{plain}	
\end{document}